%% file: root.tex
\documentclass[letterpaper, 10 pt, conference]{ieeeconf}  

\usepackage{amsmath}
\usepackage{amssymb}
\usepackage{amsfonts}
\usepackage{cite} 
\usepackage{graphicx}
\usepackage{siunitx}
\usepackage{subcaption}
\usepackage{xcolor}

\newtheorem{theorem}{Theorem}

\newtheorem{assumption}{Assumption}

\DeclareMathOperator*{\argmin}{arg\,min}

\IEEEoverridecommandlockouts                              

\title{\LARGE \bf
Control Allocation with Adaptive Augmentation for Aerodynamic Optimization of Trailing Edge Morphing Aircraft
}

\author{Mark Spiller,  Lennart Kracke, and Johannes Autenrieb
\thanks{The authors are with the
	German Aerospace Center (DLR), Institute of Flight Systems, 38108, Braunschweig, Germany. {\tt\small mark.spiller@dlr.de, lennart.kracke@dlr.de, johannes.autenrieb@dlr.de}}%
\thanks{This work has been submitted to the IEEE for possible publication.
Copyright may be transferred without notice, after which this version may no longer be accessible.}
}

\begin{document}

\maketitle
\thispagestyle{empty}
\pagestyle{empty}

\begin{abstract}

This paper presents a control allocation framework with adaptive augmentation for a trailing edge morphing aircraft, providing stability guarantees under uncertainty while exploiting the available morphing degrees of freedom to optimize aerodynamic efficiency. A baseline controller is designed using the nominal aircraft model to establish the desired closed-loop reference dynamics. For the aerodynamics, the wing shape is optimized dependent on the flight state to achieve a target elliptical lift distribution corresponding to minimum induced drag. The adaptive augmentation is tailored to the control allocation problem to account for the uncertain system dynamics and stabilize the aircraft around the reference dynamics. The resulting stabilization condition is formulated as hard constraint in the allocation problem while minimizing deviations from the corresponding state-dependent aerodynamically optimal wing shape. Numerical results demonstrate that the adaptive augmentation compensates for matched uncertainties while the control allocation optimizes the wing shape with respect to the reference elliptical lift distribution.

\end{abstract}

\section{Introduction}
\label{sec:introduction}
\input{sections/introduction.tex}

\section{Aircraft Dynamics}
\label{sec:aircarft_modeling}
\input{sections/aircarft_modeling.tex}

\section{Baseline Controller Design}
\label{sec:baseline_controller}
\input{sections/baseline_controller.tex}

\section{Aerodynamic Optimization}
\label{sec:aero_optimization}
\input{sections/aero_optimization.tex}

\section{Adaptive Augmentation and Control Allocation}
\label{sec:control_allcation}
\input{sections/control_allocation.tex}


\section{Numerical Example}
\label{sec:num_example}
\input{sections/num_example.tex}

\section{Conclusions}

This paper presents a control allocation approach with adaptive augmentation for a trailing edge morphing aircraft that accounts for uncertainty in stabilization while optimizing aerodynamic efficiency. The tracking error with respect to the desired reference behavior is guaranteed to remain bounded. Numerical results demonstrate stable tracking while exploiting the available morphing degrees of freedom to approach the target elliptical lift distribution where compatible with the required aircraft dynamics.







\bibliographystyle{IEEEtran}
\bibliography{mybibfile}

\end{document}

%% file: sections/introduction.tex
Morphing-wing aircraft offer enhanced aerodynamic performance through their ability to continuously adapt wing geometry during flight. This capability enables real-time redistribution of lift, which can be exploited to improve both aerodynamic efficiency and maneuverability. In contrast to conventional fixed-wing designs, morphing configurations allow performance optimization across a wide range of mission profiles, particularly in scenarios involving multi-regime flight, wide flight-envelope operation, and dynamically changing performance objectives, such as long-endurance surveillance combined with high-speed transit or maneuver-intensive missions.

From a control perspective, however, morphing-wing aircraft introduce significant challenges, primarily due to the inherently fast time-varying dynamics and strong couplings induced by the morphing process \cite{lingling2022design}. Control strategies of morphing wing aircraft can be classified into approaches that target the control of the aircraft dynamics, the wing shape control or both \cite{parancheerivilakkathil2024review}. This paper focuses on aircraft dynamics control under morphing and integrated approaches in which the morphing actuators are used simultaneously for flight dynamics control and wing shape adaptation.

Aircraft dynamics control: The majority of linear approaches model the morphing aircraft either as a linear switched or LPV system, and subsequently synthesize controllers based on $\mathcal{H}_\infty$ performance criteria  \cite{lee2019linear,guo2017dynamic,wang2015finite,jiang2015systematic,cheng2017non,yao2019intermittent,baldelli2008modeling}. Nonlinear control approaches include methods such as nonlinear dynamic inversion (NDI) \cite{li2016modeling}, incremental nonlinear dynamic inversion (INDI) \cite{wang2021seamless,pollak2026incremental}, backstepping \cite{tong2013multi,gong2019disturbance,wang2019morphing}, sliding mode control \cite{yan2019adaptive,dai2021longitudinal,lee2018sliding,yue2017flight}, and reinforcement learning (RL) \cite{cao2025integrated,cao2025intelligent}. The following works explicitly investigate adaptive control: In \cite{chandrasekaran2018adaptive}, an adaptive flight control system for a flapping wing aircraft is developed. Dynamic inversion in combination with PID control and adaptive augmentation is employed for longitudinal control. However, the uncertainties addressed to be compensated by the adaptive controller are not explicitly quantified, and no stability analysis is provided. Similarly, the work of \cite{cheng20231} applies dynamic inversion in combination with LQR and $\mathcal{L}_1$-based adaptive augmentation for a variable sweep morphing aircraft. The matched uncertainty assumption is adopted without further justification, and the dynamic inversion approach relies on exact knowledge of the control effectiveness matrix.

Combined flight dynamics and wing shape control: For many morphing concepts, the actuation channels for flight dynamics control and wing shape adaptation are distinct. This is the case, for example, in variable-sweep morphing \cite{lee2019linear,yao2019intermittent,wang2015finite,jiang2015systematic}, wing flapping \cite{chandrasekaran2018adaptive}, wing folding \cite{guo2017dynamic}, and telescopic wing morphing \cite{yue2017flight}. In contrast, camber morphing does not exhibit such a separation, as the same control inputs can be used for flight dynamics control as well as for wing shaping. To the best of the authors’ knowledge, \cite{wang2021seamless} is the only work that explicitly addresses this aspect. In \cite{wang2021seamless}, INDI combined with a linear error controller is applied to all servo inputs of a morphing trailing edge wing to control aerodynamic forces and structural bending moments for gust and maneuver load alleviation. The resulting wing shape is thereby determined implicitly through the control objective. The approach relies on INDI to inherently compensate for uncertainties in the control effectiveness matrix when applying the control allocation.

In this paper, a linear reduced-order representation of the lateral dynamics of a trailing edge morphing aircraft is derived from its full nonlinear dynamics. The resulting linear model incorporates parametric deviations from the nominal model and unmodeled nonlinear effects as uncertainty perturbations. The choice of states in the reduced-order model ensures that all uncertainties are matched. For the resulting over-actuated lateral dynamics, the control objective is to achieve control allocation under uncertainty while preserving stability and optimizing the wing shape for aerodynamic efficiency. Therefore, a baseline controller is designed using the nominal system knowledge to define the desired closed-loop reference dynamics. The wing shape is optimized with respect to a given reference elliptical lift distribution using the nominal aerodynamic model. We formulate a hard constraint in the control allocation based on an adaptive augmentation that preserves stability with respect to the reference model, while the allocation objective minimizes the weighted distance from the aerodynamically optimal wing shape.

Throughout the paper, deviations from trimmed conditions in states and inputs are denoted by $\delta (\cdot)=(\cdot)-(\cdot)_\mathrm{trim}$, perturbations from the nominal model are denoted by $\Delta (\cdot)=(\cdot)-(\cdot)_\mathrm{nom}$, and estimated quantities are represented by $\hat {( \cdot)}$. Positive definite matrices are expressed as $A\succ0$.

%% file: sections/aircarft_modeling.tex
The dynamics of the morphing wing aircraft are expressed in the body-fixed frame according to
\begin{subequations}
	\label{num:rigid_body_dynamics}
	\begin{align}
		\dot V_b
		&
		=
		-
		\omega \times V_b
		+
		\frac{F_b}{m},\\
		\dot \Theta
		&=
		\Psi\omega,\\
		\dot \omega &= 
		-J^{-1}(\omega \times J\omega)
		+
		J^{-1}M_b,
	\end{align}
\end{subequations}
with the states being: the body-fixed velocities
\( 
V_b=
\begin{bmatrix}
	u & v & w
\end{bmatrix}^{\top}
\in \mathbb{R}^3
\), 
the rates \( 
\omega=
\begin{bmatrix}
	p & q & r
\end{bmatrix}^{\top}
\in \mathbb{R}^3
\) 
and the Euler angles
\( 
\Theta=
\begin{bmatrix}
	\phi & \theta & \psi
\end{bmatrix}^{\top}
\in \mathbb{R}^3
\). Further, the inertia tensor is given by $J\in\mathbb{R}^{3 \times 3}$, the total mass is $m\in \mathbb{R}_{>0}$, and the Euler angle rate transformation matrix is denoted by $\Psi\triangleq\Psi(\phi,\theta)\in\mathbb{R}^{3 \times 3}$. 
The total forces $F_b\in\mathbb{R}^{3}$ and moments $M_b\in\mathbb{R}^{3}$ are defined by the functions
\begin{subequations}
	\label{num:total_forces_and_moments}
\begin{align}
	F_b
	&=
	f_{F_b}
	\left(
	\alpha, \beta, \phi, \theta, 
	\bar p, 
	\bar q, 
	\bar r, 
	Q, 
	T, \xi,\eta,\zeta
	\right),
	\\
	M_b
	&=
	f_{M_b}
	(
	\alpha, \beta, 
	\bar p, 
	\bar q, 
	\bar r,
	Q,
	\xi,\eta,\zeta
	),
\end{align}
\end{subequations}
where $\alpha\in \mathbb{R}$ is the angle of attack, $\beta\in \mathbb{R}$ is the angle of sideslip,  $\xi=\{\xi_i\}^{N}_{i=1}$ are $N$ servo inputs indexed by $\xi_i\in \mathbb{R}$ for trailing edge morphing, $\eta\in \mathbb{R}$ is the elevator deflection, $\zeta\in \mathbb{R}$ is the rudder deflection, and 
\(
Q = 
0.5 \rho_{\mathrm{air}}(h_{\mathrm{alt}}) v_{\mathrm{TAS}}^2 \notag
\)
is the dynamic pressure with airspeed $v_{\mathrm{TAS}}\in \mathbb{R}_{>0}$, air density $\rho_{\mathrm{air}}(h_{\mathrm{alt}})\in \mathbb{R}_{>0}$, and altitude $h_{\mathrm{alt}}\in\mathbb{R}_{>0}$. The normalized rates are defined as 
\(
	\bar p = \frac{B_{\mathrm{ref}}}{2 v_{\mathrm{TAS}}}
	p
\),
\(
	\bar q = 
	\frac{C_{\mathrm{ref}}}{2 v_{\mathrm{TAS}}}q
\),
\(
	\bar r = 
	\frac{B_{\mathrm{ref}}}{2 v_{\mathrm{TAS}}}r,
\)
where $C_{\mathrm{ref}}, B_{\mathrm{ref}}\in\mathbb{R}_{>0}$ are the reference chord length and span.
The thrust $T\in \mathbb{R}_{>0}$ is modeled by a look-up table of the jet engine defined by the function 
\begin{align}
	T=f_T(\tau) 
\label{num:thrust_model}
\end{align}
with throttle $\tau\in [0,1]$. Finally, the inputs of the morphing aircraft are the deflections $\xi$, $\eta$, $\zeta$ and the throttle $\tau$.

Linearizing the nonlinear dynamics (\ref{num:rigid_body_dynamics}) yields the standard state space representation
\begin{align}
	\begin{bmatrix}
		\dot r \\
		\dot \beta \\
		\dot p \\
		\dot \phi \\
	\end{bmatrix}
	=&
	\begin{bmatrix}
		N_r & N_\beta & N_p & 0 \\
		-1 & Y_\beta & 0 & \frac{g}{V_0} \\
		L_r & L_\beta & L_p & 0 \\
		\Theta_0 & 0 & 1 & 0 \\
	\end{bmatrix}
	\begin{bmatrix}
		\delta r \\
		\delta \beta \\
		\delta p \\
		\delta \phi \\
	\end{bmatrix}\notag\\
	& \qquad 
	+
	\begin{bmatrix}
		N_{\xi_1} & \hdots & N_{\xi_N} & N_\zeta \\
		0& \hdots & 0 & Y_\zeta \\
		L_{\xi_1} & \hdots & L_{\xi_N} & L_\zeta \\
		0& \hdots & 0 & 0 
	\end{bmatrix}
	\begin{bmatrix}
		\delta  \xi_1 \\
		\vdots\\
		\delta  \xi_N\\
		\delta  \zeta
	\end{bmatrix}
	\label{num:linear_lateral_motion} 
\end{align}
of lateral motion obtained at the scheduling grid points
\begin{align}
	\rho_j
	=&
	\begin{bmatrix}
		\alpha_{\mathrm{trim},j}, \beta_{\mathrm{trim},j},
		Q_{\mathrm{trim},j}
	\end{bmatrix}^{\top}.
\label{num:grid_point}
\end{align}

Due to space limitations, we restrict ourselves to the design of the lateral controller which is the one relevant for the morphing and control allocation. 

The gird points represent $j=1\dots N_{\mathrm{op}}$ trimmed operational conditions around which the nonlinear dynamics are linearized.  Using (\ref{num:grid_point}) and the scheduling vector
$
	\rho
	=
	\begin{bmatrix}
		\alpha, 
		\beta,
		Q
	\end{bmatrix}^{\top},
$ one can formulate a grid-based interpolated LPV representation of the lateral dynamics. 

After some algebra, the linearization (\ref{num:linear_lateral_motion}) can be reformulated as the reduced-order nominal lateral dynamics
\begin{align}
	\dot x
	&=
	A(\rho_j)
	\delta x
	+
	B(\rho_j)
	\delta u
	+
	E(\rho_j)
	\delta w
	\label{num:red_lat_lon_dynamics}
\end{align}
with
\begin{align}
	A
	&=
	\begin{bmatrix}
		N_r  
		-N_p\Theta_0 
		&  
		N_p 
		& 
		0  \\
		L_r  
		+
		(N_r   
		-
		L_p
		-
		N_p\Theta_0)\Theta_0
		&
		L_p 
		+
		\Theta_0 N_p
		& 
		0\\
		0  & 1 & 0 
	\end{bmatrix},
	\notag\\
	B
	&
	=
	\begin{bmatrix}
		N_{\xi_1} & \dots &N_{\xi_N} & N_\zeta  \\
		L_{\xi_1}
		+
		\Theta_0 N_{\xi_1} 
		& 
		\dots
		&
		L_{\xi_N}
		+
		\Theta_0 N_{\xi_N} 
		&  
		L_\zeta
		+
		\Theta_0 N_\zeta  \\
		0 & \dots & 0 & 0
	\end{bmatrix},
	\notag
	\\
	E
	&
	=
	\begin{bmatrix}
		N_\beta   \\
		L_\beta   
		+ 
		\Theta_0 N_\beta    \\
		0
	\end{bmatrix}
	\notag
\end{align}
and
\(
	\delta x
	=
	\begin{bmatrix}
		\delta r &
		\dot \phi &
		\delta \phi
	\end{bmatrix}^{\top}
\),
\(
	\delta u
	=
	\begin{bmatrix}
		\delta \xi_1 &
		\dots &
		\delta \xi_N &
		\delta \zeta 
	\end{bmatrix}^{\top}
\),
\(
	\delta w
	=
	\delta \beta
\).
The exogenous input $\delta w$ is measureable and $E$ denotes its input shaping matrix.

We decompose the input matrices as
\begin{alignat}{2}
	B
	&
	=
	\begin{bmatrix}
		\bar B\\
		0
	\end{bmatrix},
	&\quad
	E
	&=
	\begin{bmatrix}
		\bar E \\
		0
	\end{bmatrix}
\label{num:decomposition_input_matrices}
\end{alignat}
and make the following assumption, which has to be satisfied by nominal aircraft design to ensure sufficient control authority. 
\begin{assumption}
\label{assumpt_full_rank_input_matrix}
	The matrix $\bar{B}$ has full row rank.
\end{assumption}

The representation (\ref{num:red_lat_lon_dynamics}) describes assumed nominal dynamics that are in practice affected by parametric uncertainties as well as nonlinearities. Therefore, let $\dot x = f(x,u,w,p)$ denote the true lateral nonlinear reduced-order dynamics, where $p$ describes the true system parameters. Define 
$
z
=
\begin{bmatrix}
x^{\top} & u^{\top} & w^{\top}	
\end{bmatrix}^{\top}$
and
$
z_{0}
=
\begin{bmatrix}
	x^{\top}_{0} & u^{\top}_{0} & w^{\top}_{0}	
\end{bmatrix}^{\top}
$. Then the true dynamics can be expressed as
\begin{align}
	\dot x
	&=
	f(z_{0},p)
	+
	F
	(z-z_{0})
	+
	\mathcal{O}
	(\|
	z-z_{0}
	\|^2),	\label{num:true_nonlinear_lon_lat}
\end{align}
where the Jacobian $F=\frac{\partial f}{\partial z}|_{(z_{0},p)}$ is decomposed into known nominal parts and unknown perturbations, i.e.,
\begin{align}
F
	&=
	\begin{bmatrix}
		 A(\rho_j) + \Delta A &   
		 B(\rho_j) + \Delta B &   
		 E(\rho_j) + \Delta E
	\end{bmatrix}.
	\notag
\end{align}
By defining
\begin{align}
	&\Delta d
	=
	f(z_{0},p)
	+
	\mathcal{O}
	(\|
	z-z_{0}
	\|^2)
	+
	F
	(z_{\mathrm{trim}}-z_{0}),
	\notag
\end{align}
the true nonlinear dynamics (\ref{num:true_nonlinear_lon_lat}) can be rewritten as
\begin{align}
	\dot x
	&=
	(A(\rho_j)
	+
	\Delta A)
	\delta x
	+
	(B(\rho_j)
	+
	\Delta B)
	\delta u\notag\\
	& \qquad \quad
	+
	(E(\rho_j)
	+
	\Delta E)
	\delta w
	+
	\Delta d,
	\label{num:red_lat_lon_dynamics_pertubed}
\end{align}
where the perturbations $\Delta(\cdot)$ scale with the level of model uncertainty and the distance from the operational point $\rho_j$.

Consider the definition of the state variables $	\delta x
=
\begin{bmatrix}
	\delta r &
	\dot \phi &
	\delta \phi
\end{bmatrix}^{\top}$. Due to the choice of state variables, the dynamics of the last state are given by a pure kinematic integrator,
$
 \delta \dot{\phi} = \dot{\phi}
$. These relation is purely kinematic and therefore independent of parametric uncertainties. Consequently, all admissible perturbations preserve the zero structure in the last row and have the following form
\begin{subequations}
	\label{num:structure_of_uncertainties}
\begin{alignat}{2}
	\Delta A
	&=
	\begin{bmatrix}
		\Delta \bar A\\
		0
	\end{bmatrix},
	\qquad
	&
	\Delta B
	&=
	\begin{bmatrix}
		\Delta \bar B\\
		0
	\end{bmatrix},
	\label{num:structure_of_uncertainties_a}\\
	\Delta E
	&=
	\begin{bmatrix}
		\Delta \bar E\\
		0
	\end{bmatrix},
	\qquad
	&
	\Delta d_{i}
	&=
	\begin{bmatrix}
		\Delta \bar d_{i}\\
		0
	\end{bmatrix}.
\label{num:structure_of_uncertainties_b}
\end{alignat}
\end{subequations}

The goal of this work is to control the over-actuated uncertain system (\ref{num:red_lat_lon_dynamics_pertubed}) through control allocation while optimizing aerodynamic efficiency.

For the remainder of the paper, we mostly omit the explicit dependence of linearizations and controllers on the scheduling grid points $\rho_j$ for notational simplicity.

%% file: sections/baseline_controller.tex
In the following, we restrict ourselves to the design of local controllers at the grid points $\rho_j$. The synthesis of a global LPV controller, for instance via gain scheduling or polytopic approaches, is beyond the scope of this work.

We consider yaw rate and roll angle as control variables
\begin{align}
	\delta y
	&=
	\begin{bmatrix}
		\delta r \\
		\delta	\phi 
	\end{bmatrix}
	=
	C
	\delta x,
	\qquad
	C
	=
		\begin{bmatrix}
			1  & 0 & 0   \\
			0  & 0 & 1 
	\end{bmatrix},
	\label{num:attitude_control_varibles}
\end{align}
where the corresponding reference values are denoted as $\delta y_{r}$.

The goal is to control the system (\ref{num:red_lat_lon_dynamics_pertubed}) under uncertainty. However, only the nominal description  (\ref{num:red_lat_lon_dynamics}) is available for model-based controller synthesis. Therefore, we will first design the baseline controller based on (\ref{num:red_lat_lon_dynamics}) and subsequently formulate an adaptive augmentation to account for the uncertainties in (\ref{num:red_lat_lon_dynamics_pertubed}).

To achieve the desired control task, the linear baseline controller
\begin{align}
	\delta u_{\mathrm{base}}
	=
	K_{x} \delta  x
	+
	K_{r} \delta  y_{r}
	+
	K_{w} \delta w 
\label{num:lin_base_ctr_structure}
\end{align}
is formulated. The state feedback matrix $K_{x}$ can be designed by any linear control approach like LQR, pole placement, or some robust control method in order to achieve desired stable closed-loop dynamics of (\ref{num:red_lat_lon_dynamics}). The prefilter matrix $K_{r}$ for reference tracking is then designed as
\begin{align}
K_{r} 
=
\Gamma^{\top}(\Gamma\Gamma^{\top})^{-1}
\notag
\end{align}
with
\(
\Gamma
=
-C(A+BK_{x})^{-1}B
\)
so that $K_{r}$ always satisfies the equation for stationary accurate tracking
\begin{align}
	\delta  y
	&
	=
	-
	C
	(A
	+
	BK_{x})^{-1} 
	B
	K_{r} \delta  y_{r}
	=
	\delta  y_{r}.
	\notag
\end{align}
The compensator of the exogenous input is obtained by solving
\begin{align}
	-E
	=
	\begin{bmatrix}
		-
		\bar E\\
		0
	\end{bmatrix}
	&
	=
	\begin{bmatrix}
		\bar B K_{w}\\
		0
	\end{bmatrix}
	=
	BK_{w},
	\notag
\end{align}
for $K_{w}$, which is guaranteed to always have a solution due to Assumption \ref{assumpt_full_rank_input_matrix}. Finally, substituting the baseline controller $\delta u=\delta u_{\mathrm{base}}$ in (\ref{num:red_lat_lon_dynamics}), yields the desired closed-loop reference model
\begin{align}
	\dot x_{m}
	&=
	A_{m}
	\delta  x_{m}
	+
	B_{m}
	\delta  y_{r}.
	\label{num:desired_closed_loop_dyanmics}
\end{align}
with $A_{m}=A
+
B
K_{x}$ being Hurwitz and $B_{m}=BK_{r}$.

%% file: sections/aero_optimization.tex
In the lateral dynamics, the morphing wing introduces additional degrees of freedom that enable the controller to generate the required control moments while simultaneously optimizing the wing shape for secondary objectives beyond the primary stabilization task. In this contribution, the morphing wing shape is optimized alongside the control objective. The aim is to enhance flight efficiency in steady wings-level flight by minimizing the deviation from a reference elliptical lift distribution which serves as a proxy for minimizing induced drag.

Throughout this section, the aerodynamic optimization is performed using the nominal aerodynamic model and nominal system parameters. While this model does not fully capture the true aerodynamic behavior under uncertainty, aerodynamic optimization typically relies on the nominal model, as the uncertain aerodynamic characteristics are generally unknown a priori.

We evaluate the nominal aerodynamics model 
\begin{align}
	L_{\mathrm{req}}
	&=
	f_{L}
	\left(
	\alpha, \beta, 
	\bar p, \bar q, \bar r, 
	Q, 
	T,\xi,\eta,\zeta
	\right),
	\notag
\end{align}
to calculate the required lift force $L_{\mathrm{req}}$  that is demanded by the controller. We then define the reference segment lift force 
\begin{align}
	L_{\mathrm{ref},i}
	&=
	L_{\mathrm{req}}
	\frac{f_{\mathrm{elliptic},i}}{\sum_j f_{\mathrm{elliptic},j}}
\label{num:reference_segment_lift_force}
\end{align}
of segment $i$ based on the elliptic distribution 
\begin{align}
	f_{\mathrm{elliptic},i}
	&=
	\sqrt{1-\left(\frac{2y_i}{B_\mathrm{ref}}\right)^2},
	\notag
\end{align}
where 
\begin{align}
	y_i
	&=
	-\frac{B_\mathrm{ref}}{2}
	+
	(i-1)
	\frac{B_\mathrm{ref}}{N-1},
	\qquad
	i=1,\dots,N
	\notag
\end{align}
is the absolute spanwise coordinate of servo $i$ measured from the wing root and $B_\mathrm{ref}$ is the reference wingspan. The definition (\ref{num:reference_segment_lift_force}) guarantees that the reference segment lift force satisfies
$
\sum_i L_{\mathrm{ref},i}=L_{\mathrm{req}}$, thereby preserving the total lift demand while enforcing an elliptic spanwise distribution. 

In the following, the segment lift forces resulting from the applied control inputs are computed based on the nominal aerodynamic model. According to \cite{anderson1980numerical}, we calculate the nominal segment lift force as
\begin{align}
L_{\mathrm{nom},i}(\xi_i)
&=
\frac{1}{2}
\rho v_\mathrm{TAS}c(y_i)C_{L,i}(\alpha,\xi_i),
\notag
\end{align}
where $\xi_i$ is the deflection of servo $i$, $c(y_i)$ is the chord length at the spanwise coordinate $y_i$, $C_{L,i}$ is the sectional lift coefficient of segment $i$, and $\rho$ denotes the air density. 

The optimal servo deflections are obtained by solving the optimization problem
\begin{align}
	\{\xi_{\mathrm{aero},i}\}^N_{i=1}
	&=
	\argmin_{\{\xi_i\}^N_{i=1}}
	\sum_{i=1}^N
	\left(
	L_{\mathrm{ref},i}-L_{\mathrm{nom},i}(\xi_i)
	\right)^2,
	\notag
\end{align}
where $L_{\mathrm{ref},i}$ denotes the target elliptic lift distribution corresponding to the aerodynamically optimal condition of minimum induced drag. For the optimal rudder deflection we define $\zeta_\mathrm{aero}=0$ leading to the definition 
\begin{align}
	\delta u_{\mathrm{aero}}
	&=
	\begin{bmatrix}
		\delta \xi_{\mathrm{aero},1} &
		\dots &
		\delta \xi_{\mathrm{aero},N} &
		\delta \zeta_\mathrm{aero}
	\end{bmatrix}^{\top}
	\label{num:def_aero_opt_inputs}
\end{align}
of the optimal inputs with respect to the aerodynamics. 

%% file: sections/control_allocation.tex
To ensure lateral stabilization of the aircraft while simultaneously optimizing the wing shape, a control allocation problem is formulated. However, both stabilization and control allocation are affected by the uncertainties in the system description (\ref{num:red_lat_lon_dynamics_pertubed}). To retain stability guarantees under these uncertainties, the nominal controller is augmented by an adaptive component specifically tailored to the uncertain control allocation problem. This augmentation is incorporated as a hard constraint in the control allocation, ensuring stability under uncertain dynamics while allowing the wing shape to be optimized.

Specifically, we propose the following solution of the control allocation problem under uncertainty. We decompose the input matrix as
\begin{align}
	B
	+
	\Delta B
	&=
	\begin{bmatrix}
		\bar B
		+
		\Delta \bar B\\
		0
	\end{bmatrix}
	=
	\begin{bmatrix}
		\bar B\Lambda\\
		0
	\end{bmatrix}
	=
	B\Lambda
	\label{num:decompose_input_lambda}
\end{align}
where a particular representation of $\Lambda$ is given by
\begin{align}
	\Lambda
	&=
	I
	+
	\bar B^{\top}(\bar B\bar B^{\top})^{-1}
	\Delta \bar B
	\label{num:input_parameterization_lambda}
\end{align}
using Assumption \ref{assumpt_full_rank_input_matrix}. Further, We make the following additional assumption.
\begin{assumption}
\label{assumpt_input_paramet_lambda}
	Let $\hat \Lambda$ be an estimation of $\Lambda$. We assume that the square matrix $\Lambda$ and its estimation $\hat \Lambda$ are always non-singular.
\end{assumption}
The assumption of the non-singularity of $\Lambda$ is naturally satisfied as otherwise the aircraft would suffer from a loss of control authority. The non-singularity of $\hat \Lambda$ has to be guaranteed by the estimation algorithm. 

Defining 
\(
\tilde \Lambda
=
\hat \Lambda
- 
\Lambda
\)
yields
\begin{align}
	(B
	+
	\Delta B)
	=
	B\Lambda
	=
	B
	\hat \Lambda
	-
	B
	\tilde \Lambda.
	\label{num:input_matrix_indentity}
\end{align}
We allocate inputs $\delta u$ in such a way that 
\begin{align}
	B
	(\delta u_{\mathrm{base}}
	+
	\delta u_{\mathrm{ad}})
	=
	B
	\hat \Lambda
	\delta u
	\label{num:prop_allocation}
\end{align}
is satisfied, where $\delta u_{\mathrm{base}}$ denotes the baseline controller (\ref{num:lin_base_ctr_structure}) and $\delta u_{\mathrm{ad}}$ a yet undefined adaptive augmentation. Based on the structure of $B$ according to  (\ref{num:decomposition_input_matrices}) we can rewrite (\ref{num:prop_allocation}) as
\begin{align}
	\begin{bmatrix}
	\bar B
	(\delta u_{\mathrm{base}}
	+
	\delta u_{\mathrm{ad}})\\
	0
	\end{bmatrix}
	=
	\begin{bmatrix}
	\bar B
	\hat \Lambda
	\delta u\\
	0
	\end{bmatrix}
	\label{num:prop_allocation2}.
\end{align}
With the full row rank of $\bar B$ according to Assumption \ref{assumpt_full_rank_input_matrix} and the non-singular estimation $\hat \Lambda$ according to Assumption \ref{assumpt_input_paramet_lambda},  the allocation (\ref{num:prop_allocation2}) respectively (\ref{num:prop_allocation})  can always be achieved with regards to $\delta u$. Multiplying (\ref{num:input_matrix_indentity}) by $\delta u_{\mathrm{lat}}$ and substituting (\ref{num:prop_allocation}) leads to
\begin{align}
	(B
	+
	\Delta B
	)
	\delta u
	& 
	=
	B
	\hat \Lambda
	\delta u
	- 
	B
	\tilde \Lambda
	\delta u
	\notag\\
	& 
	=
	B
	(\delta u_{\mathrm{base}}
	+
	\delta u_{\mathrm{ad}})
	- 
	B
	\tilde \Lambda
	\delta u.
	\label{num:input_matrix_indentity2}
\end{align}
Substituting $(B+\Delta B)
\delta u$ into the open-loop dynamics (\ref{num:red_lat_lon_dynamics_pertubed}) and considering the RHS of (\ref{num:input_matrix_indentity2}) yields
\begin{align}
	\dot x
	&=
	(A
	+
	\Delta A)
	\delta x
	+
	B
	(\delta u_{\mathrm{base}}
	+
	\delta u_{\mathrm{ad}})	\notag\\
	&  \quad
	- 
	B\tilde \Lambda
	\delta u
	+
	(E
	+
	\Delta E)
	\delta w
	+
	\Delta d.
\label{num:true_open_loop_rewritten}
\end{align}
Substituting the baseline controller $\delta u_{\mathrm{base}}$ of (\ref{num:lin_base_ctr_structure}) into (\ref{num:true_open_loop_rewritten}) results in
\begin{align}
	\dot x
	&=
	A_{m}
	\delta x
	+
	B_{m} \delta  y_{r}
	+
	B
	\delta u_{\mathrm{ad}}
		\notag\\
	& \quad 
	+
	\Delta A
	\delta x
		+
	\Delta E
	\delta w
		+
	\Delta d
	- 
	B
	\tilde \Lambda
	\delta u.
	\label{num:true_closed_loop_allocation}
\end{align}
On the basis of the structure of the uncertainties in (\ref{num:structure_of_uncertainties}), the full row rank of $\bar B$ according to Assumption \ref{assumpt_full_rank_input_matrix} and the decomposition (\ref{num:decomposition_input_matrices}), it follows that the equations
\begin{align}
	\Delta A
	&=
	B\Theta_{x},
	\quad
	\Delta E
	=
	B\Theta_{w},
	\quad 
	\Delta d
	=
	B\Theta_{d},
	\label{num:matching_cond_allocation}
\end{align}
are guaranteed to have solutions for any $\Theta_{\star}$ with $\star\in\{x,w,d\}$.

Due to the changing flight conditions we treat the perturbations $\Delta (\cdot)$ as explicitly time-varying and therefore also the matching parameters $\Theta_{\star}$ and $\Lambda$. Consequently, we make the following assumption.
\begin{assumption}
\label{assumpt_time_varying_opt_gains}
	We assume the uncertainty terms to be uniformly bounded and continuously differentiable with bounded derivatives, i.e.  
	\begin{align}
		\|\Delta(\cdot,t)\|
		\le 
		\Delta_{\mathrm{max}},
		\qquad
		\|\dot \Delta(\cdot,t)\|
		\le
		\dot \Delta_{\mathrm{max}}.\notag
	\end{align}
	This implies that the matching parameters are well-defined and bounded according to
	\begin{align}
	 &\|\Lambda(t)\|
	 \le 
	 \Lambda_{\mathrm{max}},
	 \qquad
	 \|\Theta_{\star}(t)\|
	 \le
	 \Theta_{\mathrm{max}},\notag\\
	&
	\|\dot\Lambda(t)\|
	\le 
	\dot\Lambda_{\mathrm{max}},
	\qquad
	\|\dot\Theta_{\star}(t)\|
	\le
	\dot\Theta_{\mathrm{max}}
	\notag
	\end{align}
	with $\star\in\{x,w,d\}$.
\end{assumption}

Substituting the matching conditions (\ref{num:matching_cond_allocation}) into (\ref{num:true_closed_loop_allocation}) gives
\begin{align}
	\dot x
	&=
	A_{m}
	\delta x
	+
	B_{m} \delta  y_{r}
	\notag\\
	&   \quad
	+
	B
	(\delta u_{\mathrm{ad}}
	+
	\Theta_{x}
	\delta x
	+
	\Theta_{w}
	\delta w	
	+
	\Theta_{d}
	-
	\tilde \Lambda
	\delta u
	).
\label{num:true_closed_loop_allocation2}
\end{align}
For the adaptive augmentation we define
\begin{align}
	\delta u_{\mathrm{ad}}
	&=
	-
	\hat\Theta_{x}
	\delta x
	-
	\hat\Theta_{w}
	\delta w	
	-
	\hat
	\Theta_{d},
	\label{num:adpative_augemntation_input_alloc}
\end{align}
which, upon substitution into (\ref{num:true_closed_loop_allocation2}), yields
\begin{align}
	\dot x
	&=
	A_{m}
	\delta x
	+
	B_{m} \delta  y_{r}
	\notag\\
	&  \quad
	-
	B
	(
	\tilde\Theta_{x}
	\delta x
	+
	\tilde\Theta_{w}
	\delta w	
	+
	\tilde
	\Theta_{d}
	+
	\tilde \Lambda
	\delta u
	)
	\label{num:true_closed_loop_allocation3}
\end{align}
with 
$\tilde \Lambda=
	\hat \Lambda-\Lambda$ 
and
$\tilde \Theta_{\star}=
	\hat \Theta_{\star}-\Theta_{\star}$ for $\star\in\{x,w,d\}$.

To improve transient behavior during adaption, we modify the desired dynamics (\ref{num:desired_closed_loop_dyanmics}) to a calibrated closed-loop reference model (CCRM) \cite{gibson2013adaptive,gibson2015adaptive} according to
\begin{align}
	\dot x_{c}
	&=
	A_{m}
	\delta  x_{c}
	+
	B_{m}
	\delta  y_{r}
	+
	L
	e
	\label{num:CCRM_model_org}
\end{align}
where 
\begin{align}
	&
	e=\delta  x_{c}-\delta  x=  x_{c}-x
	\label{num:model_matching_error}
\end{align}
is the model matching error and $L$ is a design matrix. Note that the CCRM model (\ref{num:CCRM_model_org}) coincides with the desired closed-loop (\ref{num:desired_closed_loop_dyanmics}) when the matching error $e$ goes to zero.

From (\ref{num:true_closed_loop_allocation3}) and the CCRM model (\ref{num:CCRM_model_org}), we obtain the matching error dynamics
\begin{align}
	\dot e
	&=
	A_{c}
	e
	+
	B
	(
	\tilde\Theta_{x}
	\delta x
	+
	\tilde\Theta_{w}
	\delta w	
	+
	\tilde
	\Theta_{d}
	+
	\tilde \Lambda
	\delta u
	),
\label{num:error_dyanamics}
\end{align}
where $A_{c}=A_{m}
+
L$ is assumed Hurwitz by design of $L$. 

We can now formulate the stability guarantees for the control allocation under uncertainty.

\begin{theorem}
	\label{theo:mrac_convergence_allocation}
    Assume the applied control input $\delta u$ satisfies the constraint 
	\begin{align}
	\bar B
	(\delta u_{\mathrm{base}}
	+
	\delta u_{\mathrm{ad}})
	=
	\bar B
	\hat \Lambda
	\delta u,
	\label{num:def_alloc_problem_theo}
	\end{align}
	where $\delta u_{\mathrm{base}}$ is the baseline controller according to (\ref{num:lin_base_ctr_structure}) and $\delta u_{\mathrm{ad}}$ is the adaptive augmentation (\ref{num:adpative_augemntation_input_alloc}). Let the estimated parameters be updated as
	\begin{subequations}
	\label{num:adaptation_laws}
	\begin{align}
	\dot {\hat \Lambda}^\top 
	&=
	-
	\mathrm{Proj}
	(
	{\hat \Lambda}^{\top}, 
	\Gamma_{\Lambda} 
	\delta  u
	e^{\top}PB
	),
	\\
	\dot {\hat \Theta}_{x}^{\top} 
	&=
	-
	\mathrm{Proj}(
	{\hat \Theta}_{x}^{\top}, 
	\Gamma_{x} 
	\delta x e^{\top}PB
	),
	\\
	\dot {\hat \Theta}_{w}^{\top} 
	&=
	-
	\mathrm{Proj}(
	{\hat \Theta}_{w}^{\top},
	\Gamma_{w} \delta w e^{\top}PB
	),
	\\
	\dot {\hat \Theta}_{d}^{\top} 
	&=
	-
	\mathrm{Proj}(
	{\hat \Theta}_{d}^{\top},
	\Gamma_{d} e^{\top}PB
	),
	\end{align}
	\end{subequations}
	where $\Gamma_{\Lambda}\succ0$ and $\Gamma_{\star}\succ0$ with $\star\in\{x,w,d\}$ are tuning parameters and $P$ is the solution of the Lyapunov equation
	\begin{align}
		A_{c}^{\top}P+PA_{c}
		&=
		-Q,
		\qquad Q \succ0.
		\label{num:Lyapunov_equation}
	\end{align}
	Consider the projection operator $\mathrm{Proj}(\cdot,\cdot)$ as defined in \cite[Ch.~11]{eugene2013robust}. For $\mathrm{Proj}(\cdot,\cdot)$, we further assume that the true matching parameters
	$
	\|\Lambda(t)\|
	\le 
	\Lambda_{\mathrm{max}}$
	and
	$\|\Theta_{\star}(t)\|
	\le
	\Theta_{\mathrm{max}}$ 
	are in the interior of the convex set on which the projection operator is defined on.
	
	Then it follows that the matching error $e$ and the parameter estimation errors $\tilde \Lambda$ and
	$\tilde \Theta_{\star}$ are all uniformly ultimately bounded.
	
\end{theorem}
\begin{proof}
	By (\ref{num:def_alloc_problem_theo}) it follows that (\ref{num:prop_allocation2}) respectively (\ref{num:prop_allocation}) is satisfied. Therefore (\ref{num:input_matrix_indentity2}) holds and the lateral dynamics are (\ref{num:true_open_loop_rewritten}) respectively (\ref{num:true_closed_loop_allocation2}). As the adaptive augmentation (\ref{num:adpative_augemntation_input_alloc}) is applied it follows (\ref{num:true_closed_loop_allocation3}) and finally the error dynamics (\ref{num:error_dyanamics}) hold. 
	
Consider the Lyapunov function candidate
\begin{align}
	V
	&=
	\frac{1}{2}
	e^{\top}Pe
	+
	\frac{1}{2}\operatorname{tr}
	(\tilde \Lambda\Gamma^{-1}_{\Lambda}\tilde \Lambda^{\top})\notag\\
	&
	\qquad
	+
	\frac{1}{2}
	\sum_{\star\in\{x,w,d\}}
	\operatorname{tr}
	(
	\tilde\Theta_{\star}
	\Gamma^{-1}_{\star}
	\tilde \Theta_{\star}^{\top}
	).
\label{num:Lyapunov_fct_candidate}
\end{align}

Then, based on the error dynamics (\ref{num:error_dyanamics}) and the Lyapunov equation (\ref{num:Lyapunov_equation}), after some algebra, it follows that
\begin{align}
\dot	V
	&=
	-
	\frac{1}{2}
	e^{\top}
	Q
	e\notag\\
	&
	+
	\operatorname{tr}(
	\tilde\Theta_{x}
	(
	\delta x e^{\top}PB
	+
	\Gamma^{-1}_{x}
	\dot {\hat\Theta}_{x}^{\top}
	)
	)
	-
	\operatorname{tr}(
	\tilde\Theta_{x}
	\Gamma^{-1}_{x}
	\dot {\Theta}_{x}^{\top})
	\notag\\
	&
	+
	\operatorname{tr}(
	\tilde\Theta_{w}
	(
	\delta w e^{\top}P
	B	
	+	
	\Gamma^{-1}_{w}
	\dot {\hat\Theta}_{w}^{\top}
	)
	)
	-
	\operatorname{tr}(
	\tilde\Theta_{w}
	\Gamma^{-1}_{w}
	\dot {\Theta}_{w}^{\top}
	)
	\notag\\
	&
	+
	\operatorname{tr}(
	\tilde\Theta_{d} 
	(
	e^{\top}PB
	+
	\Gamma^{-1}_{d}
	\dot {\hat\Theta}_{d}^{\top}
	)
	)
	-
	\operatorname{tr}(
	\tilde\Theta_{d} 
	\Gamma^{-1}_{d}
	\dot {\Theta}_{d}^{\top}
	)
	\notag\\
	&
	+
	\operatorname{tr}(
	\tilde \Lambda
	(
	\delta u e^{\top}P
	B
	+
	\Gamma^{-1}_{\Lambda}
	\dot {\hat \Lambda}^{\top}
	)
	)
	-
	\operatorname{tr}(
	\tilde \Lambda
	\Gamma^{-1}_{\Lambda}
	\dot {\Lambda}^{\top}
	).
\end{align}
Using the adaptation laws (\ref{num:adaptation_laws}) and the properties of the projection operator, cf. \cite[Lemma~11.3]{eugene2013robust}, it follows
\begin{align}
	\dot V
	&\le
	-
	\frac{1}{2}
	e^{\top}
	Q
	e		
	-
	\operatorname{tr}
	(\tilde\Lambda
	\Gamma^{-1}_{\Lambda}
	\dot {\Lambda}^{\top}
	)
	-
	\sum_{\star}
	\operatorname{tr}
	(\tilde\Theta_{\star}
	\Gamma^{-1}_{\star}\dot {\Theta}^{\top}_{\star}
	)
	\label{num:theo_proof_Lyapunov_00}
	.
\end{align}
Upper bounding (\ref{num:theo_proof_Lyapunov_00}) yields
\begin{align}
	\dot V
	&\le
	-
	\frac{1}{2}
	e^{\top}
	Q
	e		
	+
	\|\bar\Lambda
	\|_F
	\dot { \bar\Lambda}_{\mathrm{max}}		
	+
	\sum_{\star}
	\|\bar\Theta_{\star}\|_F
	\dot { \bar\Theta}_{\mathrm{max}}
	\label{num:theo_proof_Lyapunov_01}
\end{align}
with
\(
\bar\Lambda
\triangleq
\tilde\Lambda
\Gamma^{-1/2}_{\Lambda}
\)
and
\(
\bar\Theta_{\star}
\triangleq
\tilde\Theta_{\star}
\Gamma^{-1/2}_{\star}
\)
as well as 
\(
\dot { \bar\Lambda}_{\mathrm{max}}
\triangleq
\dot {\Lambda}_{\mathrm{max}}
\|\Gamma^{-1/2}_{\Lambda}\|_F
\ge
\|
\dot {\Lambda}\Gamma^{-1/2}_{\Lambda}
\|_F
\)
and
\(
\dot { \bar\Theta}_{\mathrm{max}}
\triangleq
\max_{\star}
\{
\dot {\Theta}_{\mathrm{max}}
\|
\Gamma^{-1/2}_{\star}
\|_F
\}
\ge
\|
\dot {\Theta}_{\star}
\Gamma^{-1/2}_{\star}
\|_F
\)
based on Assumption~\ref{assumpt_time_varying_opt_gains}.
We can write (\ref{num:Lyapunov_fct_candidate}) as
\begin{align}
	V
	&=
	\frac{1}{2}
	e^{\top}Pe
	+
	\frac{1}{2}
	\|\bar\Lambda\|^2_F
	+
	\frac{1}{2}
	\sum_{\star}
	\|
	\bar\Theta_{\star}\|^2_F
\label{num:Lyapunov_fct_candidate_rewrite}
\end{align}
and by bounding $V$ from below we obtain
\begin{align}
	\|\bar\Lambda\|_F 
	\le
\sqrt{2V},
\qquad
\|
\bar\Theta_{\star}\|_F
\le
\sqrt{2V}.
\label{num:bounds_Frob_errors}
\end{align}
Substituting (\ref{num:bounds_Frob_errors}) into (\ref{num:theo_proof_Lyapunov_01}) results in 
\begin{align}
	\dot V
	&\le
	-
	\frac{1}{2}
\lambda_{\mathrm{min}}(Q)
\|e\|^2				
	+
	c_1
	\sqrt{V}
\label{num:Lyapunov_deriv_first_bound}
\end{align}
with $c_1=	\sqrt{2}
(
\dot { \bar\Lambda}_{\mathrm{max}}		
+
3
\dot { \bar\Theta}_{\mathrm{max}}
)$.

From Assumption~\ref{assumpt_time_varying_opt_gains}, we know that the true matching parameters $\Lambda$
and
$\Theta_{\star}$ are bounded. As we apply the projection operator, we know that the estimated parameters $\hat \Lambda$ and $\hat \Theta_{\star}$ are bounded as well. Consequently, finite bounds 
$
\tilde \Lambda_\mathrm{max}
\ge
\|\tilde \Lambda\|_F
$
and
$
\tilde \Theta_\mathrm{max}
\ge
\max_{\star}
\|\tilde \Theta_{\star}\|_F
$ 
of the parameter estimation errors exist. Therefore, 
\(
\bar\Lambda_\mathrm{max}
\triangleq
\tilde \Lambda_\mathrm{max}
\|
\Gamma^{-1/2}_{\Lambda}
\|_F
\ge
\|\bar\Lambda\|_F
\) and 
\(
\bar\Theta_{\mathrm{max}}
\triangleq
\max_{\star}
\{
\tilde \Theta_\mathrm{max}
\|\Gamma^{-1/2}_{\star}\|_F
\}
\ge
\|
\bar\Theta_{\star}\|_F
\)
are finite. Then by bounding $V$ of (\ref{num:Lyapunov_fct_candidate_rewrite}) from above, we obtain
\begin{align}
	V
	&\le
	\frac{1}{2}
	\lambda_{\mathrm{max}}(P)
	\|e\|^2
	+
	\frac{1}{2}
	\bar\Lambda^2_\mathrm{max}
	+
	\frac{3}{2}
	\bar\Theta^2_{\mathrm{max}}
\end{align}
which leads to
\begin{align}
		&-
	\frac{1}{2}
	\lambda_{\mathrm{min}}(Q)
	\|e\|^2
	\le
	-
\frac{\lambda_{\mathrm{min}}(Q)}
{\lambda_{\mathrm{max}}(P)}
V
+ c_2
\label{num:Lyapunov_fct_candidate_ub}
\end{align}
with $c_2=\frac{\lambda_{\mathrm{min}}(Q)}
{2\lambda_{\mathrm{max}}(P)}
(
\bar\Lambda^2_\mathrm{max}
+
3
\bar\Theta^2_{\mathrm{max}}
)$. Substituting (\ref{num:Lyapunov_fct_candidate_ub}) in (\ref{num:Lyapunov_deriv_first_bound}) gives
\begin{align}
	\dot V
	&\le
	-
	c_0
	V			
	+
	c_1
	\sqrt{V}
	+ 
	c_2
\label{num:Lyapunov_deriv_bound2}
\end{align}
with $c_0=\frac{\lambda_{\mathrm{min}}(Q)}
{\lambda_{\mathrm{max}}(P)}>0$ and applying Young's inequality yields
\begin{align}
	\dot V
	&\le
	-
	c_0
	V			
	+
	\frac{c_0}{2}
	V
	+
	\frac{c^2_1}{2c_0}
	+ 
	c_2
	=
	-
	\frac{c_0}{2}
	V			
	+
	\bar c
\label{num:Lyapunov_deriv_bound3}
\end{align}
with
$\bar c= 	\frac{c^2_1}{2c_0}
	+ 
	c_2$. Then
\begin{align}
	V(t)
	\le
	e^{-\frac{c_0}{2}t}V(0)
	+
	\frac{2\bar c}{c_0}
	(1
	-
	e^{-\frac{c_0}{2}t}
	)
\label{num:Lyapunov_deriv_bound4}
\end{align}
and therefore 
\begin{align}
\limsup_{t\to\infty}
	V(t)
	\le
	\frac{2\bar c}{c_0}.
\label{num:Lyapunov_deriv_UUB}
\end{align}
\end{proof}

The bound (\ref{num:Lyapunov_deriv_UUB}) can be made sufficient small if the overall uncertainty effects are small enough and if the true parameter evolve slow enough over time, then 		
$
\tilde \Lambda_\mathrm{max}
$,
$
\tilde \Theta_\mathrm{max}
$,
$
\dot {\Lambda}_{\mathrm{max}}
$,
and
$
\dot {\Theta}_{\mathrm{max}}
$
remain all small. These restriction of adaptive control in case of time-varying parameters are well known \cite{chen2018adaptive}. When explicit incorporation of time-varying effects into the adaptation process is required, the reader is referred to recent advances in  \cite{goel2024composite,gaudio2021parameter,chen2018adaptive}.

Finally, we formulate the control allocation problem 
\begin{subequations}
\label{num:def_control_allocation_problem}
\begin{align}
	&\delta u
	=
	\argmin_{\delta u\in\mathbb{R}^{N+1}} 
	\|
	\delta u
	-
	\delta u_{\mathrm{aero}}
	\|^2_{W},
	\\
	& 
	\text{s.t.}
	\qquad
	\bar B
	(\delta u_{\mathrm{base}}
	+
	\delta u_{\mathrm{ad}})
	=
	\bar B
	\hat \Lambda
	\delta u 
	\label{num:def_control_allocation_hard_constraint}
\end{align}
\end{subequations}
with weighing matrix $W\succ0$, where $\delta u_{\mathrm{aero}}$ denotes the aerodynamically optimized inputs from (\ref{num:def_aero_opt_inputs}), $\delta u_{\mathrm{base}}$ is the baseline controller (\ref{num:lin_base_ctr_structure}), $\delta u_{\mathrm{ad}}$ is the adaptive augmentation (\ref{num:adpative_augemntation_input_alloc}) and all parameters are assumed being adapted as stated by (\ref{num:adaptation_laws}). As the allocation problem enforces the hard constraint (\ref{num:def_control_allocation_hard_constraint}), it follows by Theorem \ref{theo:mrac_convergence_allocation} that stability of the matching error is preserved.

Define 
\(
	\Pi 
	=
	\bar B
	\hat \Lambda
\)
then from KKT conditions it results that
\begin{align}
	\delta u	
	&=
	\delta u_{\mathrm{aero}}
	-
	W^{-1}
	\Pi^{\top} 
	(\Pi
	W^{-1}
	\Pi^{\top})^{-1}\notag\\
	&\qquad {}\cdot 
	(
	\Pi
	\delta u_{\mathrm{aero}}
	-
	\bar B
	(\delta u_{\mathrm{base}}
	+
	\delta u_{\mathrm{ad}})
	)
\end{align}
is the analytic solution of (\ref{num:def_control_allocation_problem}), where $\Pi W^{-1}\Pi^{\top}$ is guaranteed to be non-singular under the Assumptions \ref{assumpt_full_rank_input_matrix} and \ref{assumpt_input_paramet_lambda}.

\begin{figure}
	\centering
	\begin{subfigure}{0.99\linewidth}
		\includegraphics[width=\linewidth]{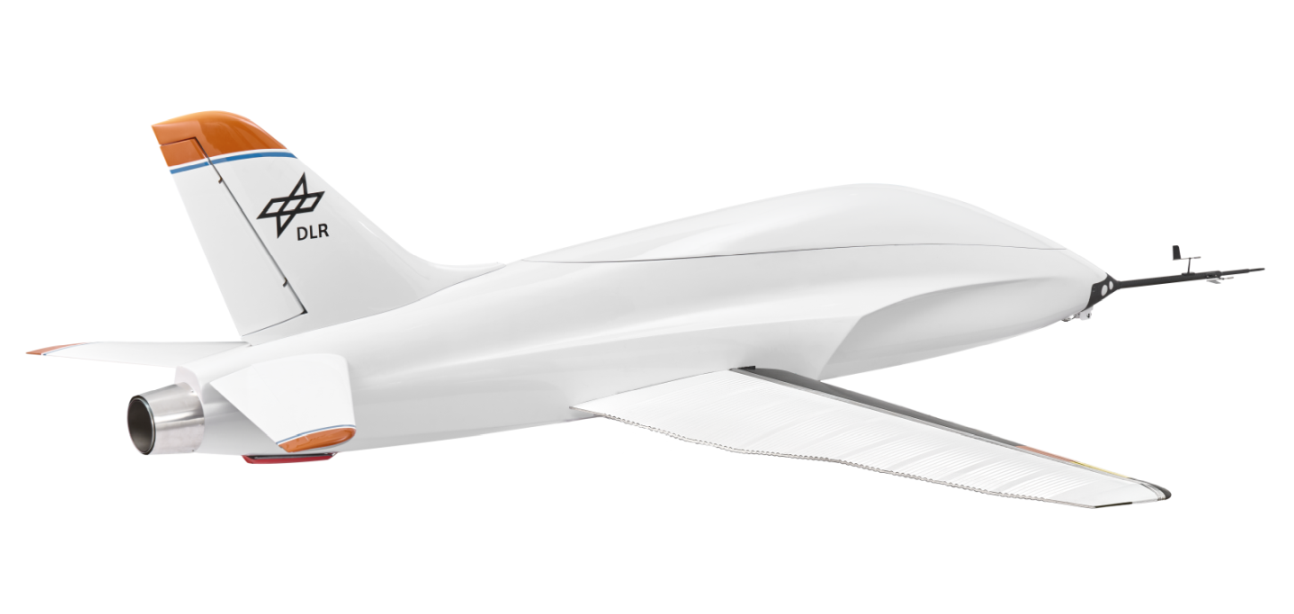}
	\end{subfigure}
	
	\caption{DLR UAV demonstrator Proteus with morphing trailing edge (Credit: DLR, CC BY-NC-ND 3.0).}
	\label{fig:num_exampl_proteus}
\end{figure}

\begin{figure*}
	\centering
	\begin{subfigure}{0.49\linewidth}
		\includegraphics[width=\linewidth]{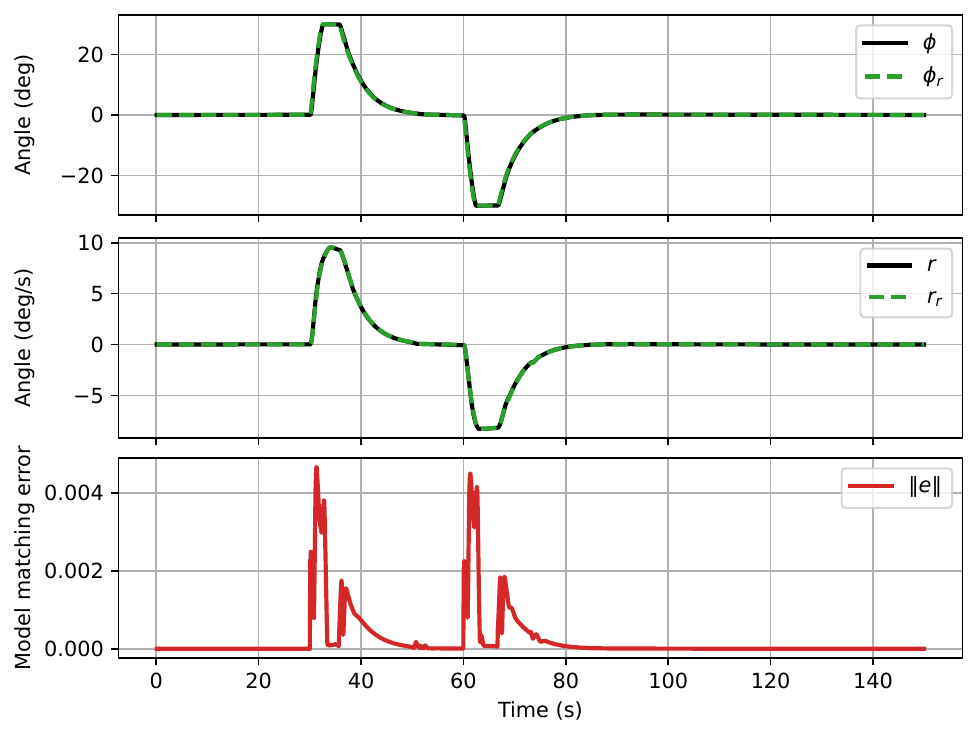}
		\caption{Tracking performance w/ augmentation}
	\end{subfigure}
	\hfill
	\begin{subfigure}{0.49\linewidth}
		\includegraphics[width=\linewidth]{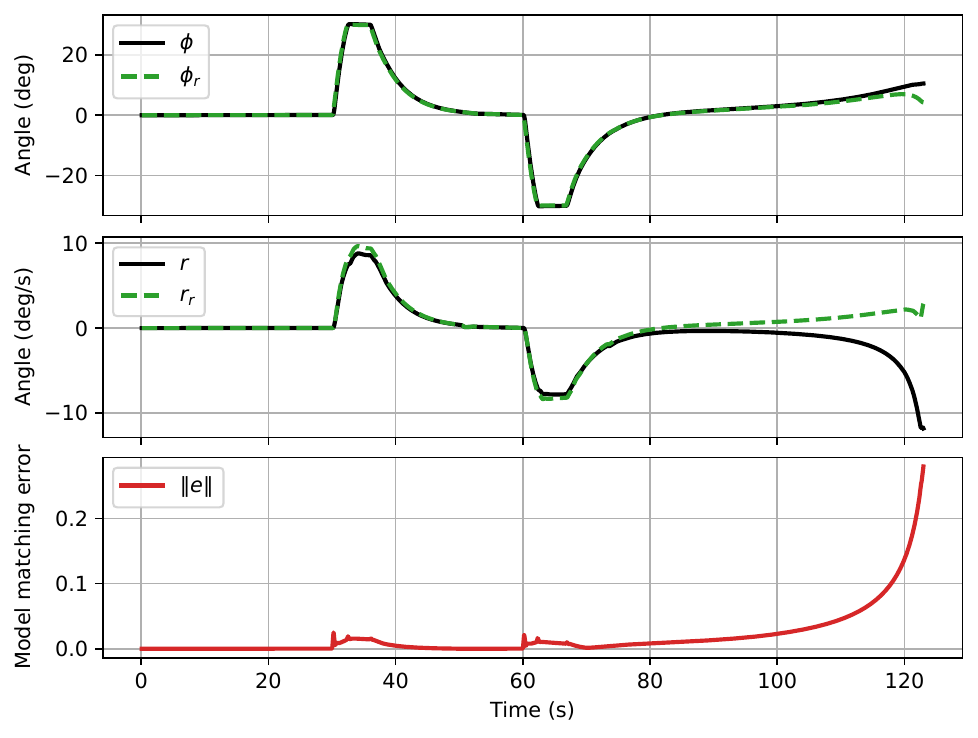}
		\caption{Tracking performance w/o augmentation}
	\end{subfigure}
	\begin{subfigure}{0.49\linewidth}
		\includegraphics[width=\linewidth]{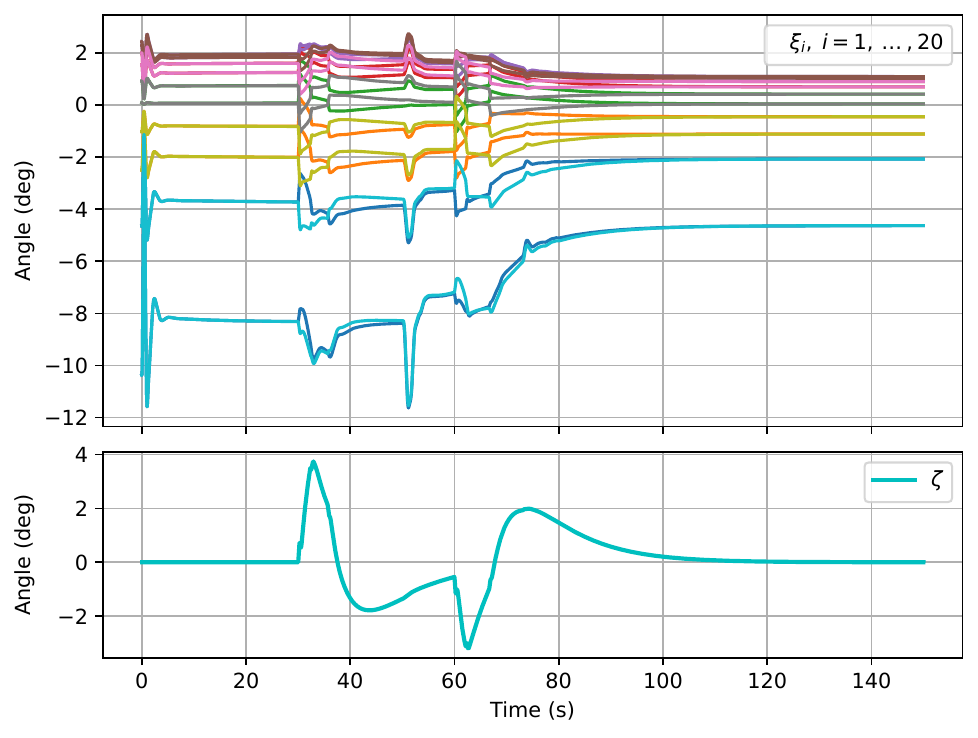}
		\caption{Control inputs w/ augmentation}
	\end{subfigure}
		\hfill	
	\begin{subfigure}{0.49\linewidth}
			\includegraphics[width=\linewidth]{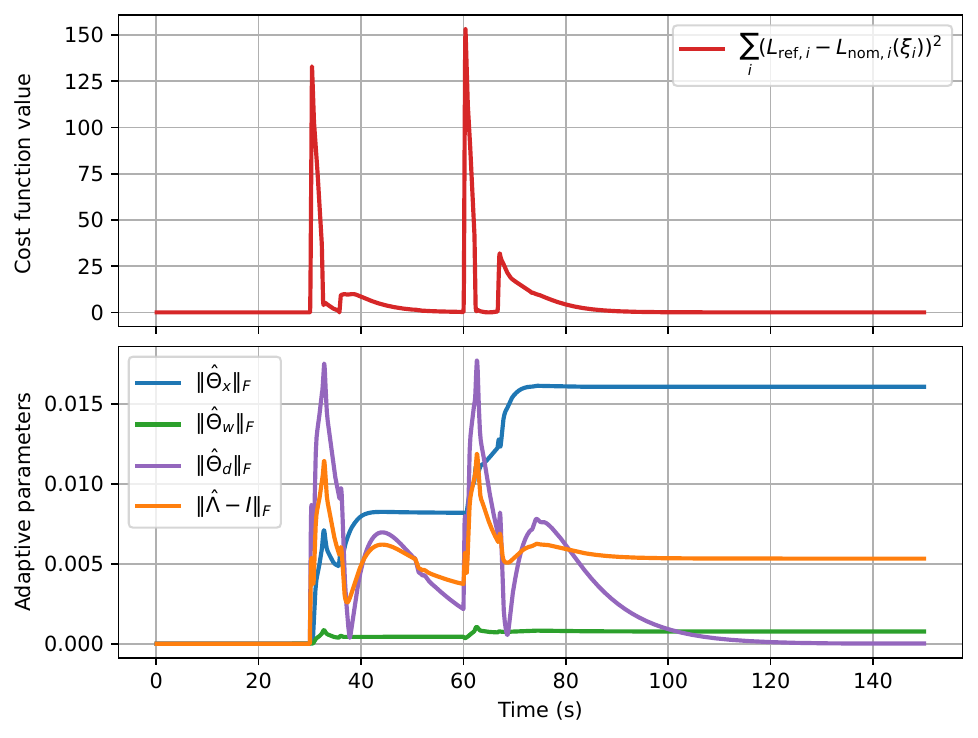}
			\caption{Aerodynamic costs and parameter adaptation w/ augmentation}
\end{subfigure}

	\caption{Performance evaluation of the proposed allocation approach with adaptive augmentation}
	\label{fig:num_exampl_sim_results}
\end{figure*}

%% file: sections/num_example.tex

The proposed allocation approach with adaptive augmentation is evaluated in a simulation environment that implements the nonlinear aircraft dynamics (\ref{num:rigid_body_dynamics})–(\ref{num:thrust_model}). The simulation model is parameterized using the morphing-wing UAV demonstrator Proteus \cite{radestock2025comparison} shown in Fig. \ref{fig:num_exampl_proteus}. The aircraft is equipped with ten servos per wing for wing trailing edge morphing.  

The initial flight condition is defined by the airspeed, altitude, flight path angle and course as $v_{\mathrm{TAS}}(t_0)=\SI{30}{m/s}$, $h(t_0)=\SI{500}{m}$, $\gamma(t_0)=\SI{0}{\deg}$, $\chi(t_0)=\SI{0}{\deg}$. Trimming the nonlinear dynamics at this operating point yields the corresponding equilibrium states and control inputs. Around the trimmed condition, the nonlinear dynamics are numerically linearized to obtain the nominal linear lateral model of (\ref{num:linear_lateral_motion}). 

The nominal aerodynamic coefficients are obtained from a precomputed look-up table generated offline using Athena Vortex Lattice (AVL) \cite{DrelaAVL}. To account for model uncertainty, the aerodynamic coefficients are perturbed. Specifically, each coefficient is modeled as
\(
C_{\mathrm{aero}} = C_{\mathrm{aero},\mathrm{nom}} ( 1 + \Delta C_\mathrm{aero} ),
\)
where $\Delta C_\mathrm{aero}$ denotes a bounded relative perturbation and $C_{\mathrm{aero},\mathrm{nom}}$ the nominal known value from the look-up table. 

The proposed lateral controller is embedded into an autopilot where it serves as an inner-loop control layer. The outer-loop controller generates the references values of the inner-loop control variables (\ref{num:attitude_control_varibles}) in order to follow desired commands for speed, course and altitude denoted as $v_{\mathrm{TAS}_c}$, $\chi_c$ and $h_c$. Specifically, we consider the following commands: 
$\chi_c(t) \in \{ 0, \tfrac{\pi}{2}, 0\}$ switching at $\SI{30}{s}$ and $\SI{60}{s}$, $h_c(t) \in \{\SI{500}{m}, \SI{700}{m}\}$ switching at $\SI{50}{s}$ and $v_{\mathrm{TAS}_c}(t) \in \{\SI{30}{m/s}, \SI{40}{m/s}\}$ switching at $\SI{50}{s}$.

The performance evaluation of the proposed approach is visualized in Fig. \ref{fig:num_exampl_sim_results}. The tracking error of the control variables $\phi$ and $r$, as well as the matching error (\ref{num:model_matching_error}) with respect to the reference model (\ref{num:CCRM_model_org}), are evaluated for both cases with and without adaptive augmentation. For the case without augmentation, we set $\delta u_{\mathrm{ad}}=0$ and $\hat \Lambda=I$, such that none of the uncertainties in (\ref{num:true_closed_loop_allocation}) are compensated. The corresponding results are shown in Fig. \ref{fig:num_exampl_sim_results} (a) and (b), respectively. It can be seen that, with adaptive augmentation, precise tracking is achieved despite the uncertainties, with only small matching errors. In contrast, without augmentation, the yaw dynamics become unstable, resulting in a rapidly increasing matching error. For the control allocation, we define the aerodynamic cost $J_{\mathrm{aero}}=\sum_i(L_{\mathrm{ref},i}-L_{\mathrm{nom},i}(\xi_i))^2$ which is shown in Fig. \ref{fig:num_exampl_sim_results}(d). It can be seen that the wing shape is aligned with the desired symmetric elliptic lift distribution in wings-level flight. Deviations occur during the turning maneuvers from $\SI{30}{s}$ to approximately $\SI{80}{s}$, as these maneuvers require an asymmetric lift distribution. This is also reflected in the left and right wing servo inputs $\xi$ shown in Fig \ref{fig:num_exampl_sim_results}(c), which are symmetric except during the interval from approximately $\SI{30}{s}$ to $\SI{80}{s}$.
Furthermore, Fig. \ref{fig:num_exampl_sim_results}(d) shows that the adaptive parameters remain bounded and smooth.